\documentclass[11pt]{amsart}
\usepackage{amsmath,amssymb,amsfonts,amsthm,amsopn}
\usepackage{latexsym,graphicx}
\usepackage{pb-diagram}

\newtheorem{theorem}{Theorem}[section]
\newtheorem{lemma}[theorem]{Lemma}
\newtheorem{corollary}[theorem]{Corollary}
\newtheorem{proposition}[theorem]{Proposition}

\newtheorem{remark}[theorem]{Remark}

\newcommand{\beqa}{\begin{eqnarray*}}
\newcommand{\eeqa}{\end{eqnarray*}}

\DeclareMathOperator*{\supp}{supp}

\newcommand{\field}[1]{\mathbb{#1}}
\newcommand{\bR}{\field{R}}        
\newcommand{\bN}{\field{N}}        
\newcommand{\bZ}{\field{Z}}        
\newcommand{\bC}{\field{C}}        
\def\cS{\mathcal{S}}

\def\rd{\bR^d}

\def\rdd{{\bR^{2d}}}

\def\R{\right)}

\def\<{\left<}
\def\>{\right>}

\def\mv1{M_v^1}

\def\mn{(m,n)}
\def\mn'{(v_1,u_2')}

\def\R{\mathbb{R}}
\def\Ren{\mathbb{R}^d}

\def\Fur{\mathcal{F}}

\def\Sn2{S_{2}(L^{2}(\Ren))}
\def\S1{S_{1}(L^{2}(\Ren))}
\def\sig00{\sigma_{0,0}}

\begin{document}
\begin{abstract} We study the propagation of singularities for semilinear Schr\"odinger equations with quadratic Hamiltonians, in particular for the semilinear harmonic oscillator. We show that the propagation still occurs along the flow the Hamiltonian flow, but for Sobolev regularities in a certain range and provided the notion of Sobolev-wave front set is conveniently modified. The proof makes use of a weighted version of the paradifferential calculus, adapted to our situation.\par
The results can be regarded as the Schr\"odinger counterpart of those known for semilinear  hyperbolic equations, which hold with the usual wave front set.
\end{abstract}

\title{Propagation of singularities for semilinear Schr\"odinger equations}

\author{Fabio Nicola  and Luigi Rodino}
\address{Dipartimento di Scienze Matematiche,
Politecnico di Torino, corso Duca degli Abruzzi 24, 10129 Torino,
Italy}
\address{Dipartimento di Matematica,
Universit\`a di Torino, via Carlo Alberto 10, 10123 Torino, Italy}

\email{fabio.nicola@polito.it}
\email{luigi.rodino@unito.it}

\maketitle
\section{Introduction}
H\"ormander \cite{hormander0} in 1991 introduced a global version of the notion of wave front set for $u$ in the space $\cS'(\rd)$ of the temperate distributions, dual to the Schwartz space $\cS(\rd)$ of the rapidly decreasing functions. Such global wave front set, $WF_G(u)$, can be easily defined in terms of the Bargman transform
\begin{equation}\label{l1}
Tu(z)=2^{-d/2}\pi^{-3d/4}\int_{\rd} e^{-iy\xi} e^{-\frac{1}{2}|x-y|^2} u(y)\, dy,
\end{equation}
where we write $z=(x,\xi)$. Namely, for $z_0\in\rdd\setminus\{0\}$ one sets $z_0\not \in WF_G(u)$, if there exists a conic neighborhood of $z_0$ in $\rdd$ where
\begin{equation}\label{l2}
|z|^N|T u(z)|<C_N\quad \textrm{for every}\ N>0,
\end{equation}
for some constant $C_N>0$, cf.\ Rodino and Wahlberg \cite{rw}. To compare with other definitions of wave front set, consider for example in dimension $d=1$ the function $u(x)=e^{i\lambda x^2/2}$ with $\lambda\in\R\setminus\{0\}$, called ``chirp'' in time-frequency analysis. This function is smooth everywhere, but 
\begin{equation}\label{l3}
WF_G\big(e^{i\lambda x^2/2}\big)=\{z=(x,\xi):\ x\not=0,\ \xi=\lambda x\}
\end{equation}
is not empty. Note also that \[
WF_G(1)=\{z=(x,\xi):\ x\not=0,\ \xi=0\},\quad WF_G(\delta)=\{z=(x,\xi):\ x=0,\ \xi\not=0\}.
\]
The definition of H\"ormander \cite{hormander0} was addressed to the study of the hyperbolic equations with double characteristics, however as a byproduct of the results of \cite{hormander0}  one may also obtain propagation of singularities for the Schr\"odinger equation 
\begin{equation}\label{l4}
D_t u+a(x,D_x)u=0,
\end{equation}
where $D_t=-i\partial_t$, and $a(x,\xi)$ is a real-valued quadratic form in $\rdd$. Namely, considering the related Hamiltonian system and flow $\chi_t:\rdd\setminus\{0\}\to\rdd\setminus\{0\}$, we have for every $t\in\R$
\begin{equation}\label{l5}
WF_G(u(t))=\chi_t(WF_G(u(0))).
\end{equation}
For example, fix attention on the quantum harmonic oscillator
\begin{equation}\label{l6}
D_t u+\frac{1}{2}(-\Delta+|x|^2)u=0,
\end{equation}
for which 
\begin{equation}\label{l7}
\chi_t(y,\eta)=
\begin{pmatrix}
(\cos t)I& (\sin t) I\\
(-\sin t)I & (\cos t)I
\end{pmatrix}
\begin{pmatrix}
y\\ \eta
\end{pmatrix}.
\end{equation}
As a test, take in \eqref{l6} the initial datum $u(0)=1$, which gives in dimension $d=1$ the solution 
\[
u(t)=c(t)e^{-i(\tan t) x^2/2}\quad {\rm for}\ t\not=\frac{\pi}{2}+k\pi,\ k\in\bZ,
\]
with $|c(t)|=|\cos t|^{-1/2}$. These are chirp functions, to which we can apply \eqref{l3}. Since $u(\pi/2+k\pi)=c_k\delta$, $|c_k|=(2\pi)^{1/2}$, from \eqref{l7} we obtain indeed \eqref{l5}, i.e.\ singularities  move along circles in the $z=(x,\xi)$ plane.\par
 Such result of propagation was generalized in different directions but, to our knowledge, always for linear equations; see Nakamura \cite{nakamura}, Wunsch \cite{wunsch}, Hassel-Wunsch \cite{hassel-wunsch}, Ito \cite{ito} and in the analytic-Gevrey category Martinez-Nakamura-Sordoni \cite{mns}, Mizuhara \cite{mizuhara} and subsequent contributions. \par
 In the present paper we want to discuss the propagation of the global wave front set for semilinear Schr\"odinger equations of the form 
 \begin{equation}\label{l8}
 D_t u+a(t,x,D_x)u=F(u),
 \end{equation}
 where $a(t,x,D)$ is a family of pseudodifferential operators with real-valued symbol $a(t,x,\xi)$ in the class of Shubin \cite{shubin}, including as a particular case the real-valued quadratic forms in $z=(x,\xi)$. We assume $F\in C^\infty(\bC)$ with $F(0)=0$.\par
 It is quite clear that the linear propagation result is lost in the semilinear case. In fact, considering again the chirp function and applying \eqref{l3} to the square we obtain a new wave front set
 \begin{equation}\label{l9}
 WF_G(u^2)=WF_G(e^{i\lambda x^2})=\{z=(x,\xi):\ x\not=0,\ \xi=2\lambda x\}.
 \end{equation}
 Starting from this obvious remark it is easy to show the appearance of anomalous singularities for the equation \eqref{l8}, even in the case $F(u)=u^2$. The role of the Hamiltonian flow can be however restored by using microlocal arguments, introduced in the years '$80$'s for the study of nonlinear hyperbolic equations, see for example \cite{bony1,bony8,bony42}. The basic idea there was that the linear propagation keeps valid if we assume that the solution $u$ belongs to a Sobolev space $H^s$ with $s$ sufficiently large and, as essential hypothesis, we limit attention to the wave front set corresponding to the regularity $H^\sigma$ with $\sigma$ sufficiently small, namely $s<\sigma<2s-d/2$. \par
 In our context, by following the approach of Taylor \cite{taylor1,taylor2}, we shall prove a similar result in the Schr\"odinger case. We will have to replace the Sobolev spaces $H^s$ with the weighted Sobolev spaces $Q^s=H^s\cap \Fur H^s$, which are more fit when dealing with operators such as the harmonic oscillator (cf.\ \cite{shubin}). As a consequence, we will also need a weighted version of the paradifferential calculus, with a combination of Littlewood-Paley decompostions in the frequency domain and in phase space.\par
 In short, the statement will be the following.
 Let us say that $z_0\in\rdd\setminus\{0\}$ does not belong to $WF^s_G(u)$, $s\in\R$, if $|z|^s T u(z)\in L^2$ in a conic neighborhood of $z_0$. Let $\chi_t$ be the Hamiltonian flow corresponding to $a(t,x,\xi)$ in \eqref{l8}. Let $d/2<s\leq \sigma<2s-d/2$ and $u\in C([0,T];Q^s)$ be a solution of \eqref{l8}. Then $z_0\not\in WF_G^\sigma(u(0))$ implies $\chi_t(z_0)\not\in WF^s_G(u(t))$. \par
 We shall also provide a preliminary result of existence and uniqueness of the Cauchy problem in the $Q^s$ frame, to give a precise setting to the propagation statement. Let us address for example to Bourgain \cite{bourgain} and Tao \cite{tao} for a survey on results of local and global existence of low regular solutions. We emphasize, however, that our results apply to classical solutions (i.e. with Sobolev regularity $s>d/2$).\par
Returning to the example of the harmonic oscillator, we may conclude propagation as described before, with $\chi_t$ as in \eqref{l7}, for the equation
 \[
 D_t u+\frac{1}{2}(-\Delta+|x|^2)u=F(u)
 \]
 independently of the nonlinearity $F(u)$. \par\medskip
 The paper is organized as follows. In Section 2 we fix some notation and we prove some preliminary estimates for Littlewood-Paley decompositions in phase space. Section 3 is devoted to the micolocal mapping property of the nonlinearity $F(u)$ in weighted Sobolev spaces, via paradifferential techniques. Finally in Section 4 we consider the evolution problem and we prove the above mentioned propagation result.

 \section{Notation and preliminary estimates}
\subsection{Notation}  The Fourier transform is normalized as
\[
\Fur f(\xi)=\widehat{f}(\xi)=\int_{\rd}e^{-ix\xi}f(x)\, dx
\]
and the pseudodifferential operator with symbol $a(x,\xi)$ is accordingly defined as 
\[
a(x,D)u=(2\pi)^{-d}\int_{\rd} e^{ix\xi} a(x,\xi)\widehat{u}(\xi)\,d\xi.
\]
\subsection{Littlewood-Paley partitions of unity \cite{taylor1,taylor2}}
Let $\psi_k(\xi)$, $k\geq0$, be a Littlewood-Paley partition of unity, therefore $\psi_0\in C^\infty_0(\rd)$ is real-valued, $\psi_0(\xi)=1$ for $|\xi|\leq 1$ and $\psi_0(\xi)=0$ for $|\xi|\geq 2$, $\psi_k(\xi)=\psi_0(2^{-k}\xi)-\psi_0(2^{-k+1}\xi)$ for $k\geq1$. In particular we see that 
\[
{\supp}\,\psi_k\subset\{2^{k-2}\leq|\xi|\leq 2^k\}
\]
 for $k\geq 1$. \par We also set 
 \[
 \Psi_k(\xi)=\sum_{j=0}^k \psi_k(\xi)=\psi_0(2^{-k}\xi),\quad k\geq0. 
 \]
By $\phi_k(x,\xi)$, $k\geq0$, we denote a similar partition of unity {\it in phase space}, and we set \[
\Phi_k(x,\xi)=\sum_{j=0}^k \phi_k(x,\xi)=\phi_0(2^{-k}x,2^{-k}\xi),\quad k\geq0.
\]
For $r>0$ we consider the Zygmund class $C^r_\ast$ endowed with the norm
\[
\|f\|_{C^r_\ast}=\sup_{j\geq0}\, 2^{rj}\|\psi_j(D)f\|_{L^\infty}.
\]
Instead, the space $C^r$, $r\geq0$, stands for the space of H\"older continuous functions of order $r$, so that $C^r=C^r_\ast$ if $r$ is not an integer, whereas $C^r\subset C^r_\ast\subset L^\infty$ if $r\in\mathbb{N}$. 
We recall from \cite[Lemma 1.3C]{taylor1} the following two lemmas.
\begin{lemma}\label{lemmachain}
Let $r>0$. There exists a constant $C>0$ such that 
\begin{equation}\label{chain}
\|g(h)\|_{C^r_\ast}\leq C \|g\|_{C^N}[1+\|h\|_{L^\infty}^N](\|h\|_{C^r_\ast}+1)]
\end{equation}
for every $g\in C^\infty$, $h\in C^r_\ast$.
\end{lemma}
\begin{lemma}\label{lemmat1}
Let $\psi\in C^\infty_0(\rd)$, $\psi(\xi)=1$ for $|\xi|\leq 1$ and $r>0$. Then the following estimates hold uniformly with respect to $0<\epsilon\leq1$:
\begin{equation}\label{lt1}
\|\psi(\epsilon D)f\|_{L^\infty}\lesssim\|f\|_{L^\infty}
\end{equation}
\begin{equation}\label{lt2}
\|\partial^\beta \psi(\epsilon D)f\|_{L^\infty}\lesssim\begin{cases}
\|f\|_{C^r}& |\beta|\leq r\\
\epsilon^{-(|\beta|-r)}\|f\|_{C^r_\ast}& |\beta|>r
\end{cases}
\end{equation}
\begin{equation}\label{lt3}
\|(I-\psi(\epsilon D))f\|_{L^\infty}\lesssim\epsilon^r\|f\|_{C^r_\ast}.
\end{equation}
\end{lemma}
We also need the following estimates for phase space localizations. 
\begin{lemma}\label{lemmat2}
Let $\phi\in C^\infty_0(\rdd)$. We have the following estimates, uniformly with respect to $0<\epsilon\leq1$:
\begin{equation}\label{lt4}
\|\phi(\epsilon x,\epsilon D) u\|_{L^\infty}\lesssim \|u\|_{L^\infty},
\end{equation}
\begin{equation}\label{lt5}
\|\partial^\beta_x \phi(\epsilon x,\epsilon D) u\|_{C^r_\ast}\lesssim \epsilon^{-|\beta|}\|u\|_{C^r_\ast},\quad r>0,\ \beta\in\bN^d.
\end{equation}
\end{lemma}
\begin{proof}
The integral kernel of the operator $\phi(\epsilon x,\epsilon D)$ is given by 
\begin{align*}
K(x,y)&=(2\pi)^{-d}\int_{\rd} e^{i(x-y)\xi}\phi(\epsilon x,\epsilon \xi)\,d\xi\\
&=\epsilon^{-d}(\Fur_2^{-1}\phi)(\epsilon x,\epsilon^{-1} (x-y)),
\end{align*}
whwre $\mathcal{F}_2$ is the partial Fourier transform.\par
Hence \eqref{lt4} follows from the estimate
\[
\sup_{x}\int|K(x,y)|\, dy\leq \sup_{x}\int_{\rd} |(\Fur_2^{-1}\phi)(\epsilon x,\eta)|\,d\eta<C
\]
with a constant $C$ independent of $\epsilon$. In fact, for every $N\geq0$ we have 
\[
|(\Fur_2^{-1}\phi)(\epsilon x,\eta)|\leq C_N (1+|\epsilon x|+|\eta|)^{-N}\leq C_N(1+|\eta|)^{-N}.
\]
Formula \eqref{lt5} for $\beta=0$ holds because the operator family $\{\phi(\epsilon x,\epsilon D):\ 0<\epsilon\leq1\}$ is bounded in H\"ormander's class $OPS^0_{1,0}$, which gives uniform boundedness on $C^r_\ast$ (see e.g.\ \cite[Proposition 2.1.D]{taylor1}).\par
In order to prove \eqref{lt5} for every $\beta$, observe that
\begin{align*}
\partial^\beta_x \phi(\epsilon x,\epsilon D) u&=(2\pi)^{-d}\sum_{\gamma\leq\beta}\binom{\beta}{\gamma}\int_{\rd} e^{ix\xi} (i\xi)^{\gamma}\epsilon^{-|\beta-\gamma|} (\partial^{\beta-\gamma}_x \phi)(\epsilon x,\epsilon \xi) \widehat{u}(\xi)\, d\xi\\
&=(2\pi)^{-d}\epsilon^{-|\beta|}
\sum_{\gamma\leq\beta}\binom{\beta}{\gamma}\int_{\rd} e^{ix\xi} (i\epsilon^{-1}\xi)^\gamma(\partial_x^{\beta-\gamma}\phi)(\epsilon x,\epsilon \xi)\widehat{u}(\xi)\, d\xi,
\end{align*}
so that it is sufficient to apply the estimate \eqref{lt5} with $\beta=0$ to the symbol $(i\xi)^\gamma$ $\partial_x^{\beta-\gamma}\phi( x, \xi)$.

\end{proof}

\subsection{Symbol classes and Sobolev spaces \cite{nicolarodino,shubin}}\par
For $0\leq \delta\leq\rho\leq 1$, $m\in\R$, we consider the space $\Gamma^m_{\rho,\delta}$ of functions $a\in C^\infty(\rdd)$ satisfying the estimates
\[
|\partial^\beta_x\partial^\alpha_\xi a(x,\xi)|
\leq C_{\alpha,\beta}(1+|x|+|\xi|)^{m-\rho|\alpha|+\delta|\beta|},\quad \alpha,\beta\in\bN^d,
\]
with the obvious Fr\'echet topology. We denote by $OP\Gamma^m_{\rho,\delta}$ the space of the corresponding pseudodifferential operators.
\par For example, the symbols $\phi_j(x,\xi)$ coming from a Littlewood-Paley partition of unity in $\rdd$ belong to a bounded subset of $\Gamma^0_{1,0}$. \par
We then consider the usual $L^2$-based Sobolev spaces $H^s=H^s(\rd)$, $s\in\R$, and define the weighted Sobolev spaces $Q^s=Q^s(\rd)$, as 
\[
Q^s=H^s\cap\Fur H^s, \quad s\geq0,
\]
and $Q^s=(Q^{-s})^\prime$ when $s<0$. In particular, $Q^0=L^2$. \par
When $s=k\in\bN$, we have the equivalence of norms
\[
\|f\|_{Q^s}\sim \sum_{|\alpha+\beta|\leq m}\|x^\alpha\partial^\beta f\|_{L^2}.
\]
It turns out that, if $0\leq\delta<\rho\leq 1$, $m,s\in\R$,
\begin{equation}\label{cont}
A\in OP\Gamma^m_{\rho,\delta}\Longrightarrow A:Q^{s+m}\to Q^s\, \textit{continuously}.
\end{equation}
In the sequel we will also use the following estimate.
\begin{lemma}\label{lemmat3}
Let $\phi_j(x,\xi)$, $j\geq0$, be a Littlewood-Paley partition of unity in $\rdd$. For every $s\in\R$ we have 
\[
\sum_{j\geq0} 2^{2js}\|\phi_j(x,D) u\|_{L^2}^2\lesssim\|u\|_{Q^s}^2.
\]
\end{lemma}
\begin{proof}
In view of the continuity result in \eqref{cont}, it is sufficient to prove that the sequence of the symbols of $\sum_{j=0}^{k}2^{2js}\phi_j(x,D)^\ast\phi_j(x,D)$ as $k\to+\infty$ converges in $\cS'(\rdd)$ to an element in $\Gamma^{2s}_{1,0}$.\par
Now, by symbolic calculus we have, for every $N\geq0$, 
\[
\phi_j(x,D)^\ast\phi_j(x,D)=a_{j,N}(x,D)+b_{j,N}(x,D)
\]
where $a_{j,N}\in\Gamma^0_{1,0}$ uniformly with respect to $j$ and is supported where $|x|+|\xi|\sim 2^j$, whereas $b_{j,N}\in \Gamma^{-N}_{1,0}$, with every seminorm $\lesssim 2^{-jN}$.\par Hence the series $\sum_{j=0}^{+\infty} 2^{2js} a_{j,N}$ converges pointwise to a symbol in $\Gamma^0_{1,0}$ and the partial sums are in a bounded subset of $\Gamma^0_{1,0}$, so that one has in fact convergence in $\cS'(\rdd)$. On the other hand, if $N>2s$ the series $\sum_{j=0}^{+\infty} 2^{2js} b_{j,N}$ is absolutely convergent with respect to every seminorm of $\Gamma^{-N}_{1,0}$. If moreover $N\geq -2s$ we have $\Gamma^{-N}_{1,0}\subset \Gamma^{2s}_{1,0}$ as well, which concludes the proof. 
\end{proof}
\subsection{Global wave front set \cite{hormander0,rw}}\label{globalregularity} Let us recall the definition of {\it global} wave front set. With respect to the Introduction, we argue here in terms of pseudodifferential operators, cf.\ \cite{hormander0,rw}. \par
A point $z_0=(x_0,\xi_0)\not=(0,0)$ is called non-characteristic for $a\in \Gamma^m_{1,0}$ if there are $\epsilon,C>0$ such that 
\[
|a(x,\xi)|\geq C(1+|x|+|\xi|)^m \quad{\rm for}\ (x,\xi)\in V_{(x_0,\xi_0),\epsilon}
\]
where $V_{z_0,\epsilon}$ is the conic neighborhood
\[
V_{z_0,\epsilon}=\Big\{z\in\rdd\setminus\{0\}:\,\Big| \frac{z}{|z|}-\frac{z_0}{|z_0|} \Big|<\epsilon,\ |z|>\epsilon^{-1}\Big\} .
\]
 Let now $ f\in\cS'(\rd)$. We define its global wave front set $WF_G(f)\subset\rdd\setminus\{0\}$ by saying that $z_0\in\rdd\setminus\{0\}$, does not belong to $WF_G(f)$ if {\it $f$ is Schwartz at $z_0$}, namely there exists $\psi\in \Gamma^0_{1,0}$ which is non-characteristic at $z_0$, such that $\psi(x,D)f\in\cS(\rd)$. The set $WF_G(f)$ is a closed conic subset of $\rdd\setminus\{0\}$. This notion of wave front set gives a characterization of the Schwartz space, in the sense that if $f\in\cS'(\rd)$ then $f\in\cS(\rd)$ if and only if $WF_G(f)=\emptyset$. \par
One can similarly define a notion of $Q^s$ wave front set $WF^s_G(f)$, $s\in\bR$, $f\in\cS'(\rd)$, by saying that $z_0\in\rdd\setminus\{0\}$, does not belong to $WF_G^s(f)$ if {\it $f$ is $Q^s$ at $z_0$}, namely, if there exists a $\psi\in \Gamma^0$ which is non-characteristic at $z_0$, such that $\psi(x,D)f\in Q^s(\rd)$. \par It is easy to see, via symbolic calculus, that if $A\in OP\Gamma^m_{\rho,\delta}$, $0\leq\delta<\rho\leq 1$, $m\in \R$, one has
\[
WF^{s}_G(Af)\subset WF^{s+m}_G(f), \quad u\in\cS'(\rd).
\]
\section{Composition and paradifferential decompositions}
Here we consider the behaviour of the Sobolev spaces $Q^s$ with respect to the composition with smooth functions. As basic fact, observe that if $u\in  Q^s\cap L^\infty$, $s\geq0$, and $F\in C^\infty(\bC)$, $F(0)=0$, then $F(u)\in Q^s$. This fact follows at once from the Moser estimates
\[
\|F(u)\|_{H^s}\leq C \|u\|_{H^s}
\]
 and 
\[
\|F(u)\|_{\Fur H^s}=\|\langle x\rangle ^s F(u)\|_{L^2}\leq C\|\langle x\rangle ^s u\|_{L^2},
\]
with $C=C(\|u\|_{L^\infty})$, where we used the Lipschitz continuity of $F$ on the range of $u$, which is bounded by assumption.  \par
If one is instead interested in the {\it microlocal} behaviour of the nonlinear operator $u\mapsto F(u)$ a deeper analysis is necessary. To this end, we now perform a suitable paradifferential decomposition in phase space (we refer the reader to \cite{taylor1,taylor2} for the classical paradifferential decomposition, which is originally performed in the frequency domain). \par
For the sake of simplicity we assume the $u\in C^0$ and $F\in C^\infty$ are real-valued, and we refer to Remark \ref{remarkcomplex} for the easy changes needed in the complex-valued case. \par

Let $\phi_k(x,\xi)$, $k\geq0$, be a Littlewood-Paley partition of unity of $\rdd$. Let $u_k=\Phi_k(x,D)u$, and consider the telescopic identity\footnote{We can assume the additional property $\phi_k(x,-\xi)=\phi_k(x,\xi)$, so that $\phi_k(x,D) u$ is real-valeud if $u$ is.}
\begin{align}
F(u)&=F(u_0)+\sum_{k=1}^{+\infty} [F(u_k)-F(u_{k-1})]\nonumber\\
&=F(u_0)+\sum_{k=0}^{+\infty} m_k(x)\phi_{k+1}(x,D)u,
\end{align}
where we set
\[
m_k(x)=\int_0^1 F'(\Phi_k(x,D)u+t\phi_{k+1}(x,D)u)\, dt.
\]
Therefore, we can write
\begin{equation}\label{ag1}
F(u)=F(u_0)+M(x,D)u
\end{equation}
with 
\[
M(x,\xi)=\sum_{k=1}^{+\infty} m_k(x)\phi_{k+1}(x,\xi).
\]
Observe that $F(u_0)\in\cS(\rd)$, since $u_0\in\cS(\rd)$ and $F(0)=0$.\par We now apply to $M(x,\xi)$ a version of the symbol smoothing technique \cite{taylor1,taylor2}, but now only in the frequency domain: for any given $\delta\in(0,1)$, we decompose further
\begin{equation}\label{ag2}
M(x,\xi)=M^\sharp(x,\xi)+M^b(x,\xi)
\end{equation}
with 
\begin{equation}\label{msharp}
M^\sharp(x,\xi)=\sum_{k=0}^{+\infty} (\psi_0(2^{-k\delta}D)m_k(x))\phi_{k+1}(x,\xi)
\end{equation}
and
\begin{equation}\label{mb}
M^b(x,\xi)=\sum_{k=0}^{+\infty} ((I-\psi_0(2^{-k\delta}D))m_k(x))\phi_{k+1}(x,\xi).
\end{equation}
The have the following symbol estimates, depending on the regularity of $u$. 
\begin{proposition}\label{pro1}
Let $u\in C^r$, $r>0$. Then $M^\sharp\in\Gamma^0_{1,\delta}$ and, more precisely, we have 
\begin{equation}\label{pro1f1}
|\partial^\beta_x \partial^\alpha_\xi M^\sharp(x,\xi)|\leq 
\begin{cases}
C_{\alpha,\beta}(1+|x|+|\xi|)^{-|\alpha|}& |\beta|\leq r\\
C_{\alpha,\beta}(1+|x|+|\xi|)^{-|\alpha|+\delta(|\beta|-r)}& |\beta|> r.
\end{cases}
\end{equation}
Moreover $M^b\in\Gamma^{-\delta r}_{1,1}$.
\end{proposition}
\begin{proof}
The estimates in \eqref{pro1f1} follow if we prove that
\begin{equation}\label{pro1f2}
\|\partial^\beta_x  (\psi_0(2^{-k\delta}D)m_k)\|_{L^\infty}\leq 
\begin{cases}
C_{\beta}& |\beta|\leq r\\
C_{\beta}2^{k\delta(|\beta|-r)}& |\beta|> r.
\end{cases}
\end{equation}
Now, by \eqref{lt2} we have 
\[
\|\partial^\beta_x  (\psi_0(2^{-k\delta}D)m_k)\|_{L^\infty}\leq 
\begin{cases}
C_{\beta}\|m_k\|_{C^r}& |\beta|\leq r\\
C_{\beta}2^{k\delta(|\beta|-r)}\|m_k\|_{C^r_\ast}& |\beta|> r.
\end{cases}
\]
Therefore it is sufficient to estimate the $C^r$ and $C^r_\ast$ norm of $m_k$ in terms of those of $u$. By the definition of $m_k$ and Lemma \ref{lemmachain} this is obtained once we have the following estimates, uniformly with respect to $t\in [0,1]$, $k\geq0$:
\begin{equation}\label{pro1f3}
\|\Phi_k(x,D)u+t\phi_{k+1}(x,D)u\|_{L^\infty}\lesssim\|u\|_{L^\infty}
\end{equation}
\begin{equation}\label{pro1f4}
\|\Phi_k(x,D)u+t\phi_{k+1}(x,D)u\|_{C^r}\lesssim\|u\|_{C^r}
\end{equation}
\begin{equation}\label{pro1f5}
\|\Phi_k(x,D)u+t\phi_{k+1}(x,D)u\|_{C^r_\ast}\lesssim\|u\|_{C^r_\ast}.
\end{equation}
Now, \eqref{pro1f3} follows by \eqref{lt4}. By the Leibniz' rule one obtains also \eqref{pro1f4} when $r$ is an integer. The formula \eqref{pro1f5} (and therefore \eqref{pro1f4} when $r$ is not an integer)  follows from \eqref{lt5}. \par
Let us now prove that  $M^b\in\Gamma^{-\delta r}_{1,1}$. It is sufficient to verify the estimates
\begin{equation}\label{pro1f6}
\|\partial^\beta_x (I-\psi_0(2^{-k\delta}D))m_k\|_{L^\infty}\lesssim 2^{-k\delta r+k|\beta|}.
\end{equation}
Now, by \eqref{lt3} we have 
\begin{equation}\label{pro1f7}
\|\partial^\beta_x (I-\psi_0(2^{-k\delta}D))m_k\|_{L^\infty}\lesssim 2^{-k\delta r}\|\partial^\beta m_k\|_{C^r_\ast},
\end{equation}
so that it remains to prove that
\begin{equation}\label{pro1f8}
\|\partial^\beta_x m_k\|_{C^r_\ast}\lesssim 2^{k|\beta|}.
\end{equation}
By \eqref{chain} (with $r$ replaced by $r+|\beta|$) and \eqref{pro1f3}, we are left to prove that 
\[
\|\partial^\beta_x(\Phi_k(x,D)+t\phi_{k+1}(x,D))u\|_{C^r_\ast}\lesssim 2^{k|\beta|}\|u\|_{C^r_\ast},
\]
which is a consequence of \eqref{lt5}.
\end{proof}

We now prove the boundedness of the ``remainder'' $M^b(x,D)$ in \eqref{mb} on the weighted Sobolev spaces $Q^s$ defined in Section 2. 
\begin{proposition}\label{pro2}
If $u\in C^r$, $r>0$, then
\[
M^b(x,D):Q^{s+\epsilon-\delta r}\to Q^s\quad \textit{continuously, for every}\ s\geq0,\ \epsilon>0. 
\] 
\end{proposition}
\begin{proof}
By the very definition of $Q^s$ we have to prove that 
\begin{equation}\label{pro2f1}
\|M^b(x,D)u\|_{H^s}\lesssim \|u\|_{Q^{s+\epsilon-\delta r}}
\end{equation}
\begin{equation}\label{pro2f1ter}
\|\langle x\rangle ^s M^b(x,D) u\|_{L^2}\lesssim \|u\|_{Q^{s+\epsilon-\delta r}}.
\end{equation}
Consider \eqref{pro2f1}. Since $\epsilon>0$ is arbitrary, we can in fact suppose $s>0$. We rewrite \eqref{mb} as
\[
M(x,\xi)=\sum_{k=0}^{+\infty} m^b_k(x)\phi_{k+1}(x,\xi),\quad m^b_k(x):=(I-\psi_0(2^{-k\delta}D))m_k,
\]
and decompose the corresponding operator as
\[
M^b(x,D)=M_1^b(x,D)+M_2^b(x,D),
\]
with
\[
M_1^b(x,D)=\sum_{k=0}^{+\infty} m^b_k(x) \psi_0(2^{-(k+1)}D)\phi_{k+1}(x,D)
\]
and 
\[
M_2^b(x,D)=\sum_{k=0}^{+\infty} m^b_k(x) (I-\psi_0(2^{-(k+1)}D))\phi_{k+1}(x,D).
\]
Now, we claim that $M^b_2(x,\xi)$ is a Schwartz symbol. One can cheek this by using the estimate
\begin{equation}\label{pro2f1bis}
\|\partial^\beta m_k^b\|_{L^\infty}\lesssim 2^{-\delta k r+k|\beta|},
\end{equation}
that is \eqref{pro1f6}, and the fact that the symbol of the operator \[
(I-\psi_0(2^{-(k+1)}D))\phi_{k+1}(x,D)
\]
 is Schwartz, with seminorms $\lesssim_N 2^{-Nk}$ for every $N\geq0$. This follows from the symbolic calculus, taking into account that $1-\psi_0(2^{-(k+1)}\xi)=0$ where $\phi_{k+1}(x,\xi)$ lives. \par
We now estimate $M_1^b(x,D)$. We further decompose it as
\begin{align}
M_1^b(x,D)u=&\sum_{k=0}^{+\infty} \sum_{j<k+5}[\psi_j(D) m^b_k(x)]\psi_0(2^{-(k+1)}D)\phi_{k+1}(x,D)u\label{pro2f2}\\
&+\sum_{j=5}^{+\infty} \sum_{k\leq j-5}[\psi_j(D) m^b_k(x)]\psi_0(2^{-(k+1)}D)\phi_{k+1}(x,D)u.\label{pro2f3}
\end{align}
Every term in the first sum of \eqref{pro2f2} has spectrum contained in the ball $|\xi|\lesssim 2^k$, so that by the classical Littlewood-Paley theory (\cite[Lemma 2.1G]{taylor1}), using $s>0$, we have
\begin{align*}
\|\sum_{k=0}^{+\infty} \sum_{j<k+5}&[\psi_j(D) m^b_k(x)]\psi_0(2^{-(k+1)}D)\phi_{k+1}(x,D)u\|^2_{H^s}\\
&\lesssim \sum_{k=0}^{+\infty} 2^{2ks}\|\sum_{j<k+5}[\psi_j(D) m^b_k(x)]\psi_0(2^{-(k+1)}D)\phi_{k+1}(x,D)u\|_{L^2}^2.
\end{align*}
Since the multiplier $\sum_{j<k+5}\psi_j(D)=\psi_0(2^{-(k+4)}D)$ is bounded on $L^\infty$ (uniformly with respect to $k$), and $\psi_0(2^{-(k+1)}D)$ is bounded on $L^2$ we can continue the estimate as
\begin{align*}
&\lesssim \sum_{k=0}^{+\infty} 2^{2ks}\|m^b_k\|_{L^\infty}^2\|\phi_{k+1}(x,D)u\|_{L^2}^2\\
&\lesssim \sum_{k=0}^{+\infty} 2^{2k(s-\delta r)}\|\phi_{k+1}(x,D)u\|_{L^2}^2\lesssim \|u\|_{Q^{s-\delta r}}^2,
\end{align*}
where we used \eqref{pro2f1bis} and, in the last estimate, Lemma \ref{lemmat3}.\par
Now, every term of the first sum in \eqref{pro2f3} has spectrum contained where $2^{j-3}\leq |\xi|\leq 2^{j+1}$, and therefore again by Littlewood-Paley theory (\cite[Lemma 2.1F]{taylor1}) we have 
\begin{align*}
\|\sum_{j=5}^{+\infty} \sum_{k\leq j-5}&[\psi_j(D) m^b_k(x)]\psi_0(2^{-(k+1)}D)\phi_{k+1}(x,D)u\|^2_{H^s}\\
&\lesssim\|\Big\{\sum_{j=5}^{+\infty}2^{2js}|\sum_{k\leq j-5}[\psi_j(D) m^b_k(x)]\psi_0(2^{-(k+1)}D)\phi_{k+1}(x,D)u|^2  \Big\}^{1/2}\|_{L^2}^2.
\end{align*}
Using \eqref{pro2f1bis} we have \[
\|\psi_j(D) m^b_k\|_{L^\infty}\lesssim_N 2^{-(j-k)N-\delta rk}\] for every $N\geq0$, so that by Young's inequality of sequences $\ell^2\ast\ell^1\hookrightarrow \ell^2$ we can continue the estimate as
\begin{align*}
&\lesssim \|\Big\{\|\sum_{j=5}^{+\infty}2^{2(s-\delta r)k}|\psi_0(2^{-(k+1)}D)\phi_{k+1}(x,D)u|^2  \Big\}^{1/2}\|_{L^2}^2\\
&= \sum_{j=5}^{+\infty}2^{2(s-\delta r)k}\|\psi_0(2^{-(k+1)}D)\phi_{k+1}(x,D)u\|_{L^2}^2\\
&\lesssim \sum_{j=5}^{+\infty}2^{2(s-\delta r)k}\|\phi_{k+1}(x,D)u\|_{L^2}^2\lesssim \|u\|_{Q^{s-\delta r}}^2.
\end{align*}
It remains to prove \eqref{pro2f1ter}. To this end, we observe that by interpolation we can suppose the $s$ is a non-negative even integer, so that $\langle x\rangle^s$ is a polynomial. Now, for $|\alpha|\leq s$, we have
\[
x^\alpha M^b(x,D)u=\sum_{\beta\leq\alpha} (-1)^{|\beta|}\binom{\alpha}{\beta}D_\xi^\beta M^b(x,D)(x^{\alpha-\beta}u).
\]
Hence 
\[
\|x^\alpha M^b(x,D)u\|_{L^2}\lesssim\sum_{\beta\leq\alpha}\|D_\xi^\beta M^b(x,D)(x^{\alpha-\beta}u)\|_{L^2}.
\]
Now, the symbol 
\[
D_\xi^\beta M^b(x,\xi)=\sum_{k=0}^{+\infty} m^b_k(x)D^\alpha_\xi \phi_k(x,\xi)
\]
has essentially the same structure than $M^b$ and the formula \eqref{pro2f1} keeps valid for it, as one can easily verify. Hence we have, for $\epsilon>0$,
\[
\|x^\alpha M^b(x,D)u\|_{L^2}\lesssim \sum_{\beta\leq\alpha} \|x^{\alpha-\beta} u\|_{H^{\epsilon-\delta r}}\lesssim\|u\|_{Q^{s+\epsilon-\delta r}},\quad\textrm{for}\, |\alpha|\leq s.
\]
This concludes the proof.
\end{proof}
\begin{remark}\label{remarkcomplex}
In the case when $F\in C^\infty(\bC)$ and $u$ are complex-valued, one has the telescopic identity 
\[
F(u)=F(\phi_0(x,D)u)+\sum_{k=0}^{+\infty} m_k(x)\phi_{k+1}(x,D)u+\sum_{k=0}^{+\infty} \tilde{m}_k(x)\phi_{k+1}(x,-D)\overline{u},
\]
where 
\[
m_k(x)=\int_0^1 \frac{\partial F}{\partial z}(\Phi_k(x,D)u+t\phi_{k+1}(x,D)u)\, dt,
\]
\[
\tilde{m}_k(x)=\int_0^1 \frac{\partial F}{\partial \overline{z}}(\Phi_k(x,D)u+t\phi_{k+1}(x,D)u)\, dt.
\]
Hence the same arguments as above can be repeated separately for the two sums. 
\end{remark}
\begin{theorem}\label{teo1}
Let $u\in C^r\cap Q^s$, $r,s>0$, and $s\leq \sigma< s+r$.  Let $F\in C^\infty(\bC)$, $F(0)=0$, and  $z_0\in T^\ast \rd\setminus\{0\}$. Then 
\[
u\in Q^\sigma\textit{ at }z_0 \Longrightarrow F(u)\in Q^\sigma\textit{ at }z_0.
\]
\end{theorem}
\begin{proof}
Suppose first the $u$ and $F$ are real-valued. In view of \eqref{ag1}, \eqref{ag2}, Propositions \ref{pro1}, \ref{pro2}, for every $\delta\in (0,1)$, $\epsilon>0$, we can write
\[
F(u)=F(\phi_0(x,D)u)+M^\sharp(x,D) u+M^b(x,D)u,
\]
with $M^\sharp(x,D)\in OP\Gamma^0_{1,\delta}$, $F(\phi_0(x,D)u)\in \cS(\rd)$ and $M^b(x,D)u\in Q^{s-\epsilon+\delta r}$. In particular, given $\sigma\in [s,s+r)$, we choose $\epsilon$ sufficiently small and $\delta$ sufficiently close to $1$ so as $\sigma\leq s-\epsilon+\delta r$. We then use that $M^\sharp(x,D)$ preserves the $Q^\sigma$ wave front set, and $Q^{s-\epsilon+\delta r}\subset Q^\sigma$.\par
In the general case, when $F$ and $u$ are complex-valued, by Remark \ref{remarkcomplex}, we have a similar decomposition, i.e.
\[
F(u)=F(\phi_0(x,D)u)+M_1^\sharp(x,D) u+M_1^b(x,D)u+M_2^\sharp(x,D) \overline{u}+M_2^b(x,D)\overline{u},
\]
where $M_1^\sharp(x,D)$, $M_2^\sharp(x,D)$ enjoy the same properties as $M^\sharp(x,D)$ and similarly for $M_1^b(x,D)$, $M_2^b(x,D)$, and we still obtain the desired conclusion.  
\end{proof}
\begin{corollary}
Let $d/2<s\leq\sigma<2s-d/2$ and $u\in Q^s$. Let $F\in C^\infty(\bC)$, $F(0)=0$, and  $z_0\in T^\ast \rd\setminus\{0\}$. Then 
\[
u\in Q^\sigma\textit{ at }z_0 \Longrightarrow F(u)\in Q^\sigma\textit{ at }z_0.
\]
\end{corollary}
\begin{proof}
It follows from Theorem \ref{teo1}, because $Q^s\subset H^s\subset C^{r}$ for every $r<s-d/2$. 
\end{proof}
\section{Propagation for semilinear Schr\"odinger equations}
We now study the propagation of $Q^s$-singularities for semilinear Schr\"odinger equations 
\[
D_tu+a(t,x,D)u=F(u)
\]
for a symbol $a(t,\cdot)\in\Gamma^2_{1,0}$, and a smooth $F\in C^\infty(\bC)$.
We first study a class of linear equations.\par
Let $T>0$ and consider the Cauchy problem 
\begin{equation}\label{cauchy}
\begin{cases}
D_t u+a(t,x,D)u+b(t,x,D)u=f(t),\\
u(0)=u_0
\end{cases}
\end{equation}
with $t\in [0,T]$, $x\in\rd$. We suppose that\medskip
\begin{itemize}
\item[{\bf (i)}] $a(t,\cdot)$ belongs to a bounded subset of $\Gamma^2_{1,0}$, $0\leq t\leq T$, and the map $t\mapsto a(t,\cdot)$ is continuous with values in $C^\infty(\rdd)$\footnote{Or equivalently in $\cS'(\rdd)$, or even pointwise.};
\item[{\bf (ii)}] ${\rm Im}\, a(t,x,\xi)\leq C$, $0\leq t\leq T$, $x,\xi\in\rd$, for some constant $C>0$;
\item[{\bf (iii)}] $b(t,\cdot)$ belongs to a bounded subset of $\Gamma^0_{1,\delta}$, $0\leq t\leq T$, for some $0<\delta<1$, and the map $t\mapsto b(t,\cdot)$ is continuous with values in $C^\infty(\rdd)$.
\end{itemize}
\medskip
\begin{theorem}\label{teo4}
Assume that {\bf (i)}--{\bf (iii)} above are fulfilled and let $s\in\R$. For every $f\in L^1((0,T);Q^s)$ and $u_0\in Q^s$, there is a unique solution $u\in C([0,T];Q^s)$ of \eqref{cauchy}.
\end{theorem}
\begin{proof}
The pattern is an easy modification of the corresponding one for hyperbolic operators (see e.g.\ \cite[Lemma 23.1.1, Theorem 23.1.2]{hormander}), so that we give just a sketch of the proof.\par First of all it is easy to see that, under the assumptions {\bf (i)} and {\bf (iii)}, the operators $a(t,x,D)$ and $b(t,x,D)$ are strongly continuous $Q^s\to Q^{s-2}$ and $Q^s\to Q^s$ respectively for every $s\in\R$ (cf.\ the argument in \cite[page 386]{hormander}).  \par
Now, by a functional analysis argument one is reduced to prove the following a priori estimates:
\begin{equation}\label{apriori}
\|u(t)\|_{Q^s}\lesssim \|u(0)\|_{Q^s}+\int_0^T \|f(t)\|_{Q^s}\, dt,\quad 0\leq t\leq T,
\end{equation}
for every $u\in C^1([0,T]; Q^s)\cap C^0([0,T]; Q^{s+2})$, with
 \[
f=D_t u+a(t,x,D)u+b(t,x,D)u.
\]
When $s=0$ this estimate follow from the identity $\frac{d}{dt}\|u(t)\|^2_{L^2}=2{\rm Re} (u'(t),u(t))$, and the sharp Garding inequality
\[
{\rm Re}(ia(t,x,D)u+ib(t,x,D)u,u)\geq -C\|u\|_{L^2}^2
\]
(see e.g.\ \cite{hormander,shubin} for the inequality for $a(t,x,D)$, whereas $b(t,x,D)$ is just an $L^2$-bounded perturbation). The case of a general real $s$ follows by applying this $L^2$ result to the operator $E_s(a(t,x,D)+b(t,x,D))E_{-s}$, where $E_s=(-\Delta+|x|^2)^{s/2}\in\Gamma^s_{1,0}$ (cf.\ \cite{shubin}), which by the symbolic calculus has the form $\tilde{a}(t,x,D)+\tilde{b}(t,x,D)$, with $\tilde{a}$ and $\tilde{b}$ satisfying the same assumptions {\bf (i)}--{\bf (iii)} as $a,b$, respectively.
\end{proof}

In order to study the propagation of singularities we strengthen the assumption ${\bf (ii)}$ as follows:\medskip
\begin{itemize}
\item[$\bf (ii)^\prime$]  $a(t,x,\xi)$ is real-valued and has a well-defined principal symbol, i.e. there exists \\ $a_2\in C([0,T];C^\infty(\rdd\setminus\{0\}))$, real-valued and positively homogeneous of degree $2$ with respect to $z=(x,\xi)$, such that, for some cut-off function $\phi\in C^\infty_0(\rdd)$, $\phi=1$ in a neighborhood of the origin, the symbol
\[
a(t,x,\xi)-\phi(x,\xi)a_2(t,x,\xi)
\]
for $0\leq t\leq T$ belongs to a bounded subset of $\Gamma^{2-\epsilon}_{1,0}$, for some $\epsilon>0$.
\end{itemize}
\medskip
Consider now the Hamiltonian system
\[
\begin{cases}\displaystyle
\dot{x}=\frac{\partial a_2}{\partial \xi}(t,x,\xi)\\ \displaystyle
\dot{\xi}=-\frac{\partial a_2}{\partial x}(t,x,\xi).
\end{cases}
\]
The Hamiltonian vector field $H_{a_2}:=(\partial a_2/\partial \xi,-\partial a_2/\partial x)$, which is smooth and positively homogeneous of degree $1$ on $\R^{2d}\setminus\{0\}$, extends to a globally Lipschitz one on $\R^{2d}$. Hence the solutions will be defined for every $t\in [0,T]$. Every integral curve corresponding to an initial condition $(x_0,\xi_0)\not=0$ is called {\it bicharacteristics} and we denote by 
\[
\chi_t:\rdd\setminus\{0\}\to \rdd\setminus\{0\}
\] the corresponding flow, which is a smooth diffeomorphism, homogeneous of degree $1$ with respect to $z=(x,\xi)\in \rdd\setminus\{0\}$, as well as its inverse. \
\begin{theorem}\label{teo5}
Assume {\bf (i)}, {$\bf (ii)^\prime$}, {\bf (iii)}. Let $u\in C([0,T];Q^{s_0})$, $s_0\in\R$, be a solution of \eqref{cauchy}, and $z_0\in\rdd\setminus\{0\}$. Then, if $s\geq s_0$ we have 
\[
u(0)\in Q^s\ \textit{at}\ z_0,\ f\in C([0,T];Q^s)\Longrightarrow u(t)\in Q^s\ \textit{at}\ \chi_t(z_0)
\]
for every $t\in (0,T]$. 
\end{theorem}
\begin{proof}
In fact, we will prove the following result. \par\medskip
{\bf Claim.} {\it For some smooth function $q_0(x,\xi)$ in $\rdd\setminus\{0\}$, positively homogeneous of degree $0$, with $q_0(z_0)\not=0$, and a cut-off function $\phi\in C^\infty_0(\rdd)$, $\phi=1$ in a neighborhood of the origin, setting $q(t,x,\xi):=\phi(x,\xi)q_0(\chi_t^{-1}(x,\xi))$ we have}
\[
q(t,x,D) u(t)\in C([0,T];Q^s).
\]
\par\medskip
By definition of $Q^s$-microlocal regularity (Section \ref{globalregularity}), this gives the desired conclusion. \par
Since we already start with a function $u\in C([0,T];Q^{s_0})$, we can argue by induction and suppose that the above claim holds for the Sobolev exponent $s-\epsilon_0$, for some fixed $\epsilon_0$, and we prove it for the exponent $s$. \par
Now, let $q_0$, $q$ as in the claim, with $q_0$ supported in a small open conic neighborhood $V\subset\rdd\setminus\{0\}$ of $z_0$, so that $q(0,x,D)u(0)\in Q^s$.
Observe that the commutator
\[
[ D_t+a(t,x,D),q(t,x,D)]
\]
has a symbol given, for large $|x|+|\xi|$,  by
\[
-i\underbrace{\Big(\frac{\partial}{\partial t}+H_{a_2}\Big)(q_0\circ \chi_t^{-1})}_{=0}+r_1=r_1, 
\]
where $r_1(t,x,\xi)$ lies in a bounded subset of $\Gamma^{-\epsilon}_{1,0}$ by {$\bf (ii)^\prime$}, and is continuous as a function of $t$ valued in $C^\infty(\rdd)$. On the other hand, we have
\[
[b(t,x,D),q(t,x,D)]=r_2(t,x,D)
\]
for some symbol $r_2(t,\cdot)$ in a bounded subset of $\Gamma^{-(1-\delta)}_{1,\delta}$, continuous as a function of $t$ valued in $C^\infty(\rdd)$.\par
Summing up, we obtained
\begin{multline*}
(D_t+a(t,x,D)+b(t,x,D))q(t,x,D)u\\
=q(t,x,D)f+r_1(t,x,D)u+r_2(t,x,D)u+ r_3(t,x,D)u.
\end{multline*}
Now, by the symbolic calculus we see that $r_1(t,x,\xi)$ and $r_2(t,x,\xi)$ have Schwartz decay away from $\chi_t(V)$, because this holds for $q(t,x,\xi)$, whereas $r_3\in C([0,T];C^\infty_0(\rdd))$. Hence, assuming the claim above for the exponent $s-\epsilon_0$, with $\epsilon_0=\min\{\epsilon, 1-\delta\}$, we see that if $V$ is small enough, 
\[
r_1(t,x,D)u+r_2(t,x,D)u+r_3(t,x,D)u\in C([0,T];Q^s).
\]
Since $q(t,x,D)f\in C([0,T]);Q^s)$ too, we deduce from Theorem \ref{teo4} that
\[
q(t,x,D)u\in C([0,T];Q^s),
\]
which gives the desired claim. 
\end{proof}

We finally come to the propagation issue for the nonlinear equation. 
\begin{theorem}
Let $a(t,x,\xi)$ satisfy the assumptions {\bf (i)} and {$\bf (ii)^\prime$} and $F\in C^\infty(\bC)$, $F(0)=0$. Let $r,s>0$ and $u\in C([0,T];C^r\cap Q^s)$ be a solution of 
\begin{equation}\label{eqnon}
D_t u+a(t,x,D)u=F(u).
\end{equation}
Then, if $z_0\in\rdd\setminus\{0\}$, for every $s\leq \sigma<s+r$, 
\[
u(0)\in Q^\sigma\ \textit{at}\ z_0\Longrightarrow u(t)\in Q^\sigma\ \textit{at}\ \chi_t(z_0)
\]
for every $t\in(0,T]$. 
\end{theorem}
\begin{proof}
As in the proof of Theorem \ref{teo1} we write 
\begin{align*}
F(u(t))=F(\phi_0(x,D)u(t))+&M_1^\sharp(t,x,D) u(t)+M_1^b(t,x,D)u(t)\\
+&M_2^\sharp(t,x,D) \overline{u(t)}+M_2^b(t,x,D)\overline{u(t)},
\end{align*}
where now $F(\phi_0(x,D)u(t))\in C([0,T];\cS(\rd))$, whereas $M_j^\sharp(t,x,\xi)$, $j=1,2$ and $M^b_j(t,x,\xi)$, $j=1,2$, lie in bounded subsets of $\Gamma^0_{1,\delta}$ and $\Gamma^{-\delta r}_{1,1}$, respectively, and are continuous as functions of $t$ with values in $C^\infty(\rdd)$.\par
Moreover by Proposition \ref{pro2} we have
\[
f(t):=M_1^b(t,x,D)u(t)+M_2^b(t,x,D)u(t)+F(\phi_0(x,D)u(t))\in C([0,T];Q^{s+\delta r-\epsilon_0})
\]
for every $\epsilon_0>0$. Hence, it is then sufficient to choose $\epsilon_0$ small enough and $\delta$ sufficiently close to $1$ so as $\sigma\leq s+\delta r-\epsilon_0$ and apply Theorem \ref{teo5} with $b(t,x,D)=-M_1^\sharp(t,x,D)-M_2^\sharp(t,x,D)$ and $f$ as above. 
\end{proof}

Using the inclusions $Q^s\subset H^s\subset C^r$ for $r<s-d/2$, we obtain the following result.
\begin{corollary}
Let $a(t,x,\xi)$ satisfy the assumptions {\bf (i)} and {$\bf (ii)^\prime$} and $F\in C^\infty(\bC)$, $F(0)=0$. Let $d/2<s\leq\sigma<2s-d/2$ and $u\in C([0,T];Q^s)$ be a solution of \eqref{eqnon}.\par
Then, if $z_0\in\rdd\setminus\{0\}$,
\[
u(0)\in Q^\sigma\ \textit{at}\ z_0\Longrightarrow u(t)\in Q^\sigma\ \textit{at}\ \chi_t(z_0)
\]
for every $t\in(0,T]$. 
\end{corollary}

\end{document}